\documentclass[12pt]{article}
\usepackage{amsthm,amssymb,amsmath,amsfonts,stmaryrd,hyperref,bbm,bm}
\usepackage[square]{natbib}
\usepackage[utf8]{inputenc}
\usepackage{hyperref}
\usepackage{xcolor}
\hypersetup{
  colorlinks,
  linkcolor={red!50!black},
  citecolor={blue!50!black},
  urlcolor={blue!80!black}
}
\usepackage{ulem}
\usepackage{tikz,pgfplots,float,subcaption}

\allowdisplaybreaks

\newcommand{\R}{\mathbb{R}}

\newcommand{\N}{\mathbb{N}}
\newcommand{\Z}{\mathbb{Z}}
\newcommand{\s}{\mathbb{S}}
\newcommand{\scal}[2]{\left\langle #1,#2 \right\rangle}

\DeclareMathOperator{\diag}{diag}
\DeclareMathOperator{\hess}{Hess}
\DeclareMathOperator{\defeg}{\overset{def}{=}}
\DeclareMathOperator{\tr}{Tr}
\DeclareMathOperator{\un}{\mathbf{1}}
\DeclareMathOperator{\ddiag}{ddiag}
\newcommand{\Rn}{\mathbb{R}^n}
\newcommand{\Rp}{\mathbb{R}^p}
\newcommand{\Rnn}{\mathbb{R}^{n \times n}}
\newcommand{\Rnp}{\mathbb{R}^{n\times p}}
\newcommand{\norm}[1]{\left\lVert#1\right\rVert}

\newtheorem{thm}{Theorem}[section]
\newtheorem{lem}[thm]{Lemma}
\newtheorem{assumption*}[thm]{Assumption}
\newtheorem{defi}[thm]{Definition}
\newtheorem*{lem*}{Lemma}
\newtheorem{prop}[thm]{Proposition}
\newtheorem*{res*}{Result}
\newtheorem{cor}[thm]{Corollary}
\newtheorem*{thm*}{Theorem}
\newtheorem*{prop*}{Proposition}
\newtheorem*{cor*}{Corollary}

\newif\ifcom
\comfalse 

\title{Benign landscape for Burer-Monteiro factorizations of MaxCut-type semidefinite programs}
\author{Faniriana Rakoto Endor\thanks{\texttt{rakotoendor@ceremade.dauphine.fr}}, Irène Waldspurger\thanks{\texttt{waldspurger@ceremade.dauphine.fr}}\\ \\
CNRS (UMR 7534), Université Paris Dauphine, Inria Mokaplan}
\date{}

\begin{document}

\maketitle

\begin{abstract}
  We consider MaxCut-type semidefinite programs (SDP) which admit a low rank solution. To numerically leverage the low rank hypothesis, a standard algorithmic approach is the Burer-Monteiro factorization, which allows to significantly reduce the dimensionality of the problem at the cost of its convexity. We give a sharp condition on the conditioning of the Laplacian matrix associated with the SDP under which any second-order critical point of the non-convex problem is a global minimizer. By applying our theorem, we improve on recent results about the correctness of the Burer-Monteiro approach on $\mathbb{Z}_2$-synchronization problems and the Kuramoto model.
\end{abstract}

\section{Introduction}
\subsection{Presentation of the problem}
\begin{sloppypar}
Semidefinite programs (SDP) are optimization tools that allow the solving and modeling of a variety of problems across applied sciences. A number of problems admits a SDP formulation in combinatorial optimization \citep*{goemans1995improved}, machine learning and data sciences \citep{lanckriet2004learning}, statistics and signal processing \citep*{candes_strohmer_voroninski}. In this paper, we are interested in so-called \textit{MaxCut-type} SDPs:
\begin{equation}\label{opt:MC} \tag{SDP}
\begin{aligned}
\underset{X\in\s^{n \times n}}{\mbox{min}} \quad & -\scal{C}{X}\\
\textrm{s.t.} \quad & X \succeq 0\\
  &\mathrm{diag}(X)=\un_n,
\end{aligned}
\end{equation}
where the operator $\diag:\R^{n \times n} \to \R^n$ extracts the diagonal of a square matrix, $\un_n=(1\dots 1)^T \in \Rn$ and the symmetric matrix $C \in \Rnn$ is called the \textit{cost matrix}. SDP of this form are especially known to provide precise convex relaxations of MaxCut problems from graph optimization \citep*{goemans1995improved}. They can also model problems such as $\Z_2$-synchronization \citep*{abbe2014decoding}, phase retrieval~\citep*{waldspurger2015phase} and the Kuramoto model \citep*{kuramoto1,ling_xu_bandeira}, for particular choices of the cost matrix $C$. Several general methods exist to numerically solve problem~\eqref{opt:MC}, but they scale poorly with $n$. For instance, interior-point solvers require $O(n^3)$ computations per iteration and $O(n^2)$ to store the variable \citep*{benson2000solving}. Other methods may have a lower per iteration complexity, but this comes at the cost of less accuracy. For instance, in the first-order method \citep*{o2016conic},  each iteration costs at worse $O(n^2)$, but $O(1/\varepsilon)$ iterations are required to reach $\varepsilon$-accuracy.

To reduce the computational complexity of solvers, one must exploit the specific properties of the problem at hand, if any. For instance, it may be known in advance that the solution to~\eqref{opt:MC} is low-rank: \citep*{pataki} guarantees that there exists a solution with rank bounded by $\sqrt{2n}+O(1)$ and, when~\eqref{opt:MC} is the relaxation of a combinatorial problem, the optimal rank is often much less (see for instance \citep*{candes_strohmer_voroninski} for a theoretical justification in a particular case, \citep*{journee_bach_absil_sepulchre} for a numerical investigation).
  In this case, it is possible to tackle the problem using its so-called \textit{Burer-Monteiro factorization} \citep*{burer2003nonlinear}. The principle is to factor the variable as $X = VV^T$, for $V\in\R^{n\times p}$, where $p\in\N$ is larger or equal to the rank of the sought solution, and much smaller than $n$. Then, one optimizes over $V$, instead of directly over $X$:
\begin{equation}\label{opt:MCF}\tag{Burer-Monteiro}
\begin{aligned}
\underset{V \in \R^{n \times p}}{\mbox{min}} \quad & -\scal{C}{VV^T}\\
\textrm{s.t.} \quad & \diag(VV^T) = \un_n.\\
\end{aligned}
\end{equation}
This factorized problem can be tackled with a number of standard algorithms. For instance, \citep*{burer2003nonlinear} uses an augmented Lagrangian approach, while \citep*{journee_bach_absil_sepulchre} proposes a manifold based second-order method.

In the factorized problem, the number of variables is reduced to $np$ instead of $O(n^2)$ in the initial problem, which is computationally advantageous when $p\ll n$. However, the convexity is lost, so standard solvers are not guaranteed to reach the solution. Still, in practice, they oftentimes converge to a global solution $V \in \Rnp$ of the factorized problem, for which $X=VV^T$ solves the initial problem.

In this article, we focus on the simplest case, where Problem~\eqref{opt:MC} has a rank $1$ solution $X = \bm{xx}^T$, for $\bm{x} \in \{ \pm 1\}^n$, and try to understand under which conditions one can certify that standard solvers converge to a global solution. The rank $1$ hypothesis is satisfied in applications like $\mathbb{Z}_2$-synchronization and the Kuramoto model.

\subsection{Prior work and our contribution}

The main explanation proposed in the literature for the success of standard algorithms at solving~\eqref{opt:MCF} has been the \textit{benign non-convexity} of the optimization landscape: it may be that all second-order critical points of~\eqref{opt:MCF} are global minimizers. Since standard algorithms typically find a second-order critical point \citep*{lee_panageas_piliouras_simchowitz_jordan_recht}, they consequently find a global minimizer.

Literature suggests that the greater $p$ is, the more likely it is that the landscape is benign. More precisely, when $p \ge \sqrt{2n} + O(1)$, the landscape of the factorized problem is benign for almost all cost matrices $C$ \citep*{boumal2020deterministic}. This property is even true for all cost matrices if $p> \frac{n}{2}$ \citep*[Cor. 5.11]{boumal2020deterministic}, while it can fail for a zero Lebesgue measure subset of cost matrices if $\sqrt{2n} + O(1) \leq p \leq \frac{n}{2}$ \citep*{Liam}. However, when $p \le \sqrt{2n} + O(1)$, there is a subset of cost matrices $C$ of positive Lebesgue measure for which~\eqref{opt:MCF} admits non-optimal critical points \citep*{WW20} (with a gap to the optimal value scaling in $O(1/p)$ \citep*{Mei2017SolvingSF}, but strictly positive).

Nonetheless, in practice, standard algorithms seem to find a solution of~\eqref{opt:MCF} below the threshold $\sqrt{2n}$, suggesting that, maybe, the set of cost matrices with a non-optimal critical point is small, and ``typical'' cost matrices do not belong to it. Therefore, researchers have tried to find properties on $C$ guaranteeing that $C$ is not in this bad set, focusing for the moment on the setting where the minimizer of~\eqref{opt:MC} has rank $1$.
The articles \citep*{mcrae2024benign} and \citep*{mcrae2024nonconvex} discuss matrices $C$ with a specific structure, motivated by synchronization problems. They prove that the landscape of~\eqref{opt:MCF} is benign under conditions which involve eigenvalues of the subcomponents of $C$. \citep*{ling2023local} considers general matrices $C$ and shows that the landscape is benign if the condition number of the associated \textit{Laplacian matrix} is smaller than $\frac{p-1}{2}$.
For important instances of~\eqref{opt:MC} (mainly $\mathbb{Z}_2$-synchronization with additive Gaussian noise and multiplicative Bernoulli noise), these recent results show that standard algorithms, applied to~\eqref{opt:MCF}, retrieve the rank $1$ solution under close to optimal conditions.
\end{sloppypar}

\paragraph{Main result}
Our main result is a sufficient condition on the condition number of the Laplacian matrix of~\eqref{opt:MC} which ensures that the landscape of~\eqref{opt:MCF} is benign. This tightens the result of \citep*{ling2023local}: we show that if the condition number is less than $p$ (instead of $\frac{p-1}{2}$ in \citep*{ling2023local}), then the landscape of~\eqref{opt:MCF} is benign. Importantly, we show that this bound is optimal. Finally, by applying our theorem to $\mathbb{Z}_2$-synchronization and the Kuramoto oscillator system, we also improve on the applications of \citep*{mcrae2024nonconvex} and \citep*{ling2023local}.

Just as we were finishing the journal version of this article, the preprint \citep*{mcrae_normalized} was uploaded on arXiv. It extends our main theorem by accounting for a diagonal preconditioner. This allows to derive improved versions of our applications, matching the correctness bounds for convex semidefinite relaxations (without Burer-Monteiro factorization).

\subsection{Structure of the paper}
In section~\ref{section2}, we present our main result and, in section~\ref{section3}, its application to $\mathbb{Z}_2$-synchronization with additive Gaussian, then Bernoulli noise, and finally to the synchronization of the Kuramoto model. In section~\ref{section4} we provide the proof of the main theorem, which consists in a novel reformulation of our main result as finding an appropriate dual certificate for a specific convex minimization problem. We also make a comparison between our strategy of proof and that of \citep*{ling2023local} by reinterpreting the author's proof as finding a dual certificate. Most technical details will be left in the appendix \ref{Appendix}.
\subsection{Notation}
Throughout this paper, $\s^{n \times n}$ is the set of symmetric $n \times n$ matrices. We write $X \succeq 0$ if $X$ is a positive semidefinite matrix. For a matrix $X \in \Rnn$, when it makes sense, $\lambda_1(X) \le \lambda_2(X) \le \dots \le \lambda_n(X)$ are the eigenvalues of $X$ in ascending order. For matrices $X,Y \in \R^{n \times m}$, $\scal{X}{Y}=\tr(X^TY)$ is the standard inner product on $\R^{n \times m}$, $X \odot Y$ is the entry-wise or Hadamard product, $\norm{X}_F=\sqrt{\scal{X}{X}}$ is the Frobenius norm on $\R^{n \times m}$, $X_{i:}\in\R^m$ is the $i$-th row of $X$ and $X_{:j}\in\Rn$ is the $j$-th column of $X$. For $X \in \R^{n \times m}$, $\norm{X}$ is the spectral or $\ell_2$ operator norm of $X$ and $\norm{X}_\infty$ is the $\ell_\infty$-norm of $X$ i.e. the maximum entry in absolute value. The operator $\ddiag:\Rnn\to\s^{n\times n}$ zeroes out all the non diagonal entries of a matrix and for any vector $x\in\Rn$, $\diag(x)\in\s^{n\times n}$ is the diagonal matrix with the coordinates of $x$ on the diagonal. For any $x,y \in \R$ the notation $x \lesssim y$ means that there exists a constant $C > 0$ that does not depend on any parameter, such that $x \le C y$. For any vector $x\in\Rn$, $\norm{x}$ is the Euclidean norm of $x$, $\un_n=(1\dots 1)^T\in\Rn$.
\section{Main result}\label{section2}

Problem~\eqref{opt:MCF} can be seen as minimizing a function over the product of spheres
\begin{align*}
  \{V \in \Rnp,\diag(VV^T)=\un_n\} 
  & =\{V\in\Rnp,\norm{V_{1:}}=\dots=\norm{V_{n:}}=1\} \\
  & = (\s^{p-1})^n.
\end{align*}
The set $(\s^{p-1})^n$ can be endowed with the Riemannian structure inherited from that of $\Rnp$. It is then a Riemannian manifold.
\begin{defi}
  Let $\mathcal{M}$ be a Riemannian manifold and $f:\mathcal{M}\to\R$ a twice-differentiable function.
  For any $x\in\mathcal{M}$, we say that
  \begin{itemize}
  \item $x$ is a first-order critical point if $\nabla f(x)=0$, where $\nabla f(x)$ is the Riemannian gradient of $f$ at $x$ (which belongs to the tangent space $T_x\mathcal{M}$);
  \item $x$ is a second-order critical point (SOCP) if $\nabla f(x)=0$ and $\operatorname{Hess} f(x)\succeq 0$, where $\operatorname{Hess}f(x)$ is the Riemannian Hessian of $f$ at $x$ (which is a bilinear map on $T_x\mathcal{M}$).
  \end{itemize}
\end{defi}
More detailed explanations of these concepts can be found in \citep*{absil2008optimization} or \citep*{boumal2023introduction}.

Let us recall that we will study~\eqref{opt:MC} under the assumption that it has a rank one solution $X=\bm{xx}^T$. To set up the statement of the theorem, let us consider $\bm{x}\in\{\pm 1\}^n$, and define the Laplacian matrix as
\begin{equation}\label{defL}
  \bm{L} = \ddiag(C\bm{x}\bm{x}^T) - C.
\end{equation}
Note that, by construction, $\bm{Lx}=0$. Standard duality theory shows that $\bm{xx}^T$ is a (rank $1$) solution to~\eqref{opt:MC} if $\bm{L}\succeq 0$; this solution is unique if, in addition, $\lambda_2(\bm{L})>0$.

Our theorem gives a sufficient condition on the condition number $\frac{\lambda_n(\bm{L})}{\lambda_2(\bm{L})}$ of the Laplacian matrix under which all SOCP of~\eqref{opt:MCF} are optimal.

\begin{thm}\label{thm1}
  Fix a cost matrix $C \in \s^{n \times n}$ and a binary vector $\bm{x}\in\{\pm 1\}^n$. Assume that $\bm{L} \succeq 0$ and $\lambda_2(\bm{L})>0$. If
  \begin{equation*}
    p > \frac{\lambda_n(\bm{L})}{\lambda_2(\bm{L})},
  \end{equation*}
  then any second-order critical point $V$ of~\eqref{opt:MCF} is a global minimizer, i.e. $VV^T = \bm{x}\bm{x}^T$.
\end{thm}
In particular, if the condition number of the Laplacian matrix is upper bounded by $p$, then standard optimization algorithms converge to a global minimum of the factorized problem. This result is purely deterministic and holds for a variety of cost matrices $C$ without assumption on their structure. It improves on \citep*[Theorem 2.1]{ling2023local}, which reads as follows.
\begin{thm*}[\citep*{ling2023local}]
    Under the same assumptions as in theorem~\ref{thm1}, assume that
    \begin{equation*}
        p \ge \frac{2\lambda_n(\bm{L})}{\lambda_2(\bm{L})}+1.
    \end{equation*}
    Then all second-order critical points of~\eqref{opt:MCF} are optimal.
\end{thm*}
Our results are similar in nature but the proofs are quite different.
Moreover, our bound is optimal in the sense of the following property, the proof of which can be found in the appendix~\ref{proof_sec_2}.
\begin{prop}\label{propopti}
    Let $p\ge2$ and $n \ge 6p$. If $n$ or $p$ is even, there exist $C\in\s^{n \times n},\bm{x}\in\{\pm 1\}^n$ satisfying the assumptions of theorem~\ref{thm1} such that $\frac{\lambda_n(\bm{L})}{\lambda_2(\bm{L})}=p$ and problem~\eqref{opt:MCF} admits a non optimal second-order critical point.
\end{prop}
\section{Applications}\label{section3}
\subsection{$\mathbb{Z}_2$-synchronization with additive Gaussian noise}
Here, we consider the $\mathbb{Z}_2$-synchronization problem with additive Gaussian noise which consists in reconstructing a binary vector $\bm{x}$ with coordinates $x_1,\dots,x_n \in \{\pm1\}$ from  noisy measurements $x_ix_j + \sigma W_{ij}$ where $ W_{ij} =  W_{ji} \sim \mathcal{N}(0,1)$, $W_{ii}=0$ and $\sigma > 0$. This problem admits a relaxation of the form~\eqref{opt:MC} with cost matrix
\begin{equation}\label{costZ2synch1}
  C = \bm{xx^T} + \sigma W.
\end{equation}

\citep*{bandeira2018random} shows that this SDP relaxation retrieves the rank $1$ matrix $\bm{xx^T}$ when $\sigma < \sqrt{\frac{n}{(2+\varepsilon)\log n}}$ (for any $\varepsilon>0$), and explains that, for larger values of $\sigma$, no algorithm is expected to succeed.
Using our theorem~\ref{thm1}, we can show that the more tractable Burer-Monteiro factorization reaches the same threshold up to a multiplicative factor which goes to $1$ when $p$ becomes large.
\begin{cor}\label{cor1}
    We consider the $\mathbb{Z}_2$-synchronization problem with Gaussian noise, where the cost matrix is defined by~\eqref{costZ2synch1}. For any $\varepsilon > 0$ and large enough $n$, if
    \begin{equation}\label{ineqZ2synch1}
        \sigma < \frac{p-1}{p+1}\sqrt{\frac{n}{(2+\varepsilon)\log n}},
    \end{equation}
    then all second-order critical points of~\eqref{opt:MCF} are optimal with probability at least $1-n^{-\varepsilon/4}-4e^{-n}$.
\end{cor}
This corollary is proved in~\ref{proof_sec_3}. It improves on \citep*[corollary 2.4]{ling2023local}, which reads as follows.
\begin{cor*}[\citep*{ling2023local}]
    Under the same conditions as corollary~\ref{cor1}, if
    \begin{equation*}
        \sigma < \frac{p-3}{4(p+1)}\sqrt{\frac{n}{\log n}},
    \end{equation*}
    then all SOCP of the factorized problem~\eqref{opt:MCF} are optimal with high probability.
\end{cor*}
Indeed, our result holds for $p$ as small as $2$ whereas theirs needs $p \ge 4$. In the large $p$ limit, our bound is better by a constant multiplicative factor. We also improve on \citep*[Corollary 1]{mcrae2024nonconvex}.
\begin{cor*}[\citep*{mcrae2024nonconvex}]
    For $n \ge 2$, $\varepsilon > 0$, if the noise level of cost matrix~\eqref{costZ2synch1} satisfies
    \begin{equation*}
        \sigma \le \frac{p-3}{p-1}\sqrt{\frac{n}{(2+\varepsilon)\log n}},
    \end{equation*}
    then all second-order critical points of~\eqref{opt:MCF} are optimal with probability $\to 1$ as $n\to\infty$.
  \end{cor*}
  
The improvement lies in the fact that our result does no prohibit us from taking $p$ as small as $2$. The proof is built upon tools used both in \citep*{ling2023local} and \citep*{mcrae2024nonconvex}.
\subsection{$\mathbb{Z}_2$-synchronization with Bernoulli noise}
The problem of $\mathbb{Z}_2$-synchronization with Bernoulli noise consists in recovering a binary vector $\bm{x} \in \{\pm 1\}^n$ from its pairwise observations $x_ix_j$, where the sign of $x_ix_j$ is flipped with probability $\frac{1-\delta}{2}$, for some $0<\delta\le1$. In other words, when $\delta$ is close to $1$, the signs are not flipped and when $\delta$ is close to $0$, the observations are often corrupted. This leads to a problem of the form~\eqref{opt:MC}, with
\begin{equation}\label{Cber}
    C_{ij} = \left\{
    \begin{array}{ll}
        x_ix_j & \mbox{with probability } \frac{1+\delta}{2} \mbox{ if } i \neq j, \\
        -x_ix_j & \mbox{with probability } \frac{1-\delta}{2}\mbox{ if } i \neq j, \\
        0 & \mbox{ if } i = j.
    \end{array}
\right.
\end{equation}
Our result gives a condition on $\delta$ under which the landscape of the factorized problem~\eqref{opt:MCF} is benign.
\begin{cor}\label{Z2ber}
    We consider the $\mathbb{Z}_2$-synchronization problem with Bernoulli noise of parameter $0< \delta\le 1$, where the cost matrix is defined by~\eqref{Cber}. For any $\varepsilon>0$ and large enough $n$, if
    \begin{equation}\label{ineqZ2ber}
      \delta > \frac{p+1}{p-1} \sqrt{\frac{(2+\varepsilon) \log n}{n}},
    \end{equation}
    then all second-order critical points of~\eqref{opt:MCF} are optimal with probability at least $1-n^{-3}-n^{-\frac{\varepsilon}{3}}$.
\end{cor}
The proof of this corollary is in~\ref{proof_sec_3}. This corollary is an improvement on \citep*[Theorem 2]{mcrae2024nonconvex} in the case where the observations $x_ix_j$ are complete. This theorem reads as follows.
\begin{thm*}[\citep*{mcrae2024nonconvex}]
    Consider the $\mathbb{Z}_2$-synchronization problem with Bernoulli noise for some $0<\delta\le1$. Assume that $p \ge 4$ and there exists some $\varepsilon>0$ such that
    \begin{equation*}
        \delta > \frac{p-1}{p-3} \sqrt{\frac{(2+\varepsilon) \log n}{n}}
    \end{equation*}
    Then, with probability $\to 1$ as $n\to \infty$, all second-order points of the factorized problem~\eqref{opt:MCF} with cost matrix as in~\eqref{Cber} are optimal.
\end{thm*}
First of all, our result does not prevent us from taking $p$ as small as $2$. Furthermore, our bound on $\delta$ is better in the regime when $p$ stays constant, but $n$ is large. Our proof of this theorem builds on ideas found in \citep*{ling2023local} and \citep*{mcrae2024nonconvex}.

\subsection{The Kuramoto model}
The homogeneous Kuramoto model \citep*{kuramoto1} is a system of $n$ oscillators on the sphere $\s^1$ identified with their phase $\theta_1, \dots, \theta_n$ evolving according to the following dynamics: for each $i\leq n$,
\begin{equation}\label{kuras1}
    \left\{
    \begin{array}{ll}
       \dot{\theta}_i(t) = \sum_{ j = 1}^n C_{ij} \sin(\theta_j(t)-\theta_i(t)) \mbox{ for } t \ge 0, \\
        \dot{\theta_i}(0)=\theta_i^0 \in \s^1.
    \end{array}
\right.
\end{equation}
The matrix $C$ is called a coupling matrix. If $C_{ij} > 0$ (resp. $C_{ij} <0$), the oscillators $\theta_i$ and $\theta_j$ are attractive (resp. repulsive). The Kuramoto model also exists on higher dimensional spheres $\s^{p-1}$, in which it describes the behavior of a system of $n$ oscillators $V=(V_1, \dots,V_n)^T \in (\s^{p-1})^n$, which evolve by maximizing the following energy function
\begin{equation*}
    \mathcal{E}(V_1,\dots,V_n) \defeg \sum_{i,j=1}^nC_{ij}\scal{V_i}{V_j}=\scal{C}{VV^T}.
\end{equation*}
In other words, the Kuramoto model describes the following (ascending) gradient flow dynamic on the product of spheres $(\s^{p-1})^n$
\begin{equation}\label{GDKura}
    \left\{
    \begin{array}{ll}
       \dot{V}(t) = \nabla \mathcal{E}(V) \mbox{ for } t \ge 0, \\
        V(0)=V^0\in(\s^{p-1})^n,
    \end{array}
\right.
\end{equation}
where the gradient is taken in the Riemannian sense on $(\s^{p-1})^n$. Note that model \eqref{kuras1}, which is the most studied in the literature, corresponds to the case $p = 2$ by parameterizing $\s^{p-1}=\s^1$ as $V_i=(\cos \theta_i, \sin \theta_i)$.
\begin{sloppypar}
One of the most interesting questions arising in the study of the Kuramoto model is whether or not the oscillators synchronize at infinity no matter the initial condition i.e. $\norm{V_i(t)-V_j(t)} \to 0$ as $t \to \infty$ for any $i,j$ and any generic initial condition $V^0 \in (\s^{p-1})^n$. A very recent work \citep*{jain2025random}, which extends results in \begin{raggedright}\citep*{abdalla2022expander}\end{raggedright}, indicates that if the matrix $C$ is drawn according to an Erdős–Rényi distribution, then the system \eqref{kuras1} is globally synchronizing if the underlying graph is connected.
\end{sloppypar}

To tackle the question of synchronization, we adopt the perspective of landscape analysis. We recall that, when the gradient flow of an analytic function converges, its limit is a first-order critical point of that function. Moreover, thanks to the center-stable manifold theorem, for a generic initial condition, the limit is a second-order critical point. Note that there is synchronization if and only if the limit is of the form $V=\un_n \bm{z}^T$ for some $\bm{z} \in \s^{p-1}$ (all rows of $V$ are equal to $\bm{z}^T$) or equivalently $VV^T = \un_n \un_n^T$. Therefore, to show that \eqref{GDKura} synchronizes, it is enough to have two properties:
\begin{enumerate}
    \item any maximizer $V$ of the energy function $\mathcal{E}$ satisfies $VV^T = \un_n \un_n^T$,
    \item all second-order critical points of $-\mathcal{E}$ are optimal.
\end{enumerate}
This is equivalent to the fact that $\un_n\un_n^T$ is an optimal rank $1$ solution to~\eqref{opt:MC}, and any SOCP of \eqref{opt:MCF} is globally optimal.

In the following, will focus on the Kuramoto model with possible repulsion between oscillators as studied in \citep*{ling2023local} with the following coupling matrix
\begin{equation}\label{Ckur}
    C_{ij} = \left\{
    \begin{array}{ll}
        ~~~1 & \mbox{with probability } 1-\alpha \mbox{ if } i \neq j, \\
        -1 & \mbox{with probability } \alpha  \mbox{ if } i \neq j, \\
        ~~~0 & \mbox{ if } i = j,
    \end{array}
\right.
\end{equation}
where $\alpha \in [0,1/2)$ is the intensity of repulsion : two oscillators are attractive with probability $1-\alpha$ and repulsive with probability $\alpha$. Note that it is possible to adapt the model and our proof to account for the absence of interaction between particles $i$ and $j$, by imposing $C_{ij}=0$ with some probability as is often the case when one studies the Kuramoto model. Our result on the benign landscape of the Kuramoto model on $\s^1$ is the following.

\begin{cor}\label{Kuracor}
    Consider the Kuramoto model on $\s^1$ (Equation~\eqref{kuras1}) with initial condition $\theta_1^0, \dots, \theta_n^0$ distributed uniformly at random on the sphere and a coupling matrix $C$ as in \eqref{Ckur} for some $\alpha \in [0, 1/2)$. For any $\varepsilon > 0$, and $n$ large enough, if
    \begin{equation*}
        \alpha < \frac{1}{2}-\frac{3}{2}\sqrt{\frac{(2+\varepsilon)\log n}{n}},
    \end{equation*}
    then the oscillators synchronize with probability at least $1-n^{-3}-n^{-\frac{\varepsilon}{3}}$.
\end{cor}
This is a particular case of synchronization of the Kuramoto model on the sphere $\s^{p-1}$, the proof of which can be found in ~\ref{proof_sec_3}. It improves on \citep*[corollary 2.3]{ling2023local}, which excludes synchronization on $\s^1$.

\section{Proof of the main theorem}\label{section4}
Throughout the proof, we fix a symmetric cost matrix $C \in \s^{n \times n}, \bm{x} \in \{\pm1\}^n$ such that the associated Laplacian matrix $\bm{L}=\ddiag(C\bm{xx}^T)-C$ is positive semidefinite and $\lambda_2(\bm{L})>0.$

Moreover, we assume without loss of generality that $\bm{x}=\un_n$ so that the rank one solution of~\eqref{opt:MC} is $X_*=\un_n\un_n^T$ and a solution of the factorized problem is $V_*=\frac{1}{\sqrt{p}}\un_n\un_p^T$. Indeed, if the solution of~\eqref{opt:MC} is $\bm{xx}^T$, then the solutions of the factorized problem are all vectors of the form $V_* = \bm{x}\bm{z}^T$ for $\bm{z}\in\s^{p-1}$ a unit vector. The change of variable $V \mapsto \diag(\bm{x})V$ does not affect the landscape of the factorized problem and changes solutions into $V_* = \un_n\bm{z}^T$. We refer to \citep*{mcrae2024benign} for more information on this change of variable.

Furthermore, note that due to the diagonal constraint, changing the diagonal of the cost matrix does not change the landscape of the problems. As such, we can replace the cost matrix $C$ with
\begin{equation*}
  C-\diag(C\un_n)=-\bm{L}
\end{equation*}
and study the problem
\begin{equation}\label{opt:MCL}
\begin{aligned}
\underset{X\in\s^{n \times n}}{\mbox{min}} \quad & \scal{\bm{L}}{X}\\
\textrm{s.t.} \quad & X \succeq 0\\
  &\mathrm{diag}(X)=\un_n,
\end{aligned}
\end{equation}
and its factorized form
\begin{equation}\label{opt:MCLF}
\begin{aligned}
\underset{V \in \R^{n \times p}}{\mbox{min}} \quad & \scal{\bm{L}}{VV^T}\\
\textrm{s.t.} \quad & \diag(VV^T) = \un_n.\\
\end{aligned}
\end{equation}
Similar changes were made in the proofs of \citep*{mcrae2024benign}. Note that we have by assumption $\bm{L} \succeq 0$, $\bm{L}\un_n=0$ and $\lambda_2(\bm{L}) > 0$.

\subsection{Formulas for the gradient and Hessian}
Before proving theorem~\ref{thm1}, we provide explicit formulas for the Riemannian gradient and Hessian of the cost function of~\eqref{opt:MCLF}. First, recall that the tangent space of the manifold $(\s^{p-1})^n$ at $V$ is
\begin{equation*}
    T_V(\s^{p-1})^n = \{\dot{V}\in\R^{n \times p}:\diag(\dot{V}V^T)=0\},
\end{equation*}
Let $P_{T_V}:\Rnp\to T_V(\s^{p-1})^n$ be the orthogonal projection onto the tangent space, i.e. $P_V(X)=X-\ddiag(XV^T)V$ for $X\in\Rnp$.

The gradient of the objective function in~\eqref{opt:MCF} at a point $V \in (\s^{p-1})^n$ is
\begin{equation}\label{FOCP}
  2\left(\bm{L}-\ddiag\left(\bm{L}VV^T\right)\right)V.
\end{equation}
In particular, $V$ is first-order critical if and only if $\left(\bm{L}-\ddiag\left(\bm{L}VV^T\right)\right)V=0$.

The Hessian at $V$ is
\begin{equation}\label{eq_hess}
\hess_V:\dot{V}\in T_V(\s^{p-1})^n \mapsto 2 P_{T_V}\left(\left(\bm{L}-\ddiag\left(\bm{L}VV^T\right)\right)\dot{V}\right)
\end{equation}
If $V$ is first-order critical, it is second-order critical if and only if $\hess_V$ is positive semidefinite. As $P_{T_V}$ is self-adjoint, that is equivalent to the fact that for all $\dot{V} \in T_V(\s^{p-1})^n$,
\begin{equation}\label{socp3}
    \scal{\hess_V(\dot{V})}{\dot{V}}=2\scal{\left(\bm{L}-\ddiag\left(\bm{L}VV^T\right)\right)\dot{V}}{\dot{V}} \ge 0.
\end{equation}
  
\subsection{Variational formulation}

We prove the contrapositive of the main theorem : if $V$ is a non optimal SOCP then the condition number of the Laplacian matrix is at least $p$. Let us fix $V \in (\s^{p-1})^n$ which is second-order critical, but not optimal for~\eqref{opt:MCLF}, i.e. $VV^T \neq \un_n \un_n^T$.

In this subsection, we rephrase the problem of showing that the condition number of the Laplacian matrix is at least $p$ as looking for a feasible point of a convex problem with large enough objective value. As will be explained in subsection~\ref{ss:remarks}, our proof admits a more concise formulation which does not necessitate to explicitly introduce the convex problem. However, the rephrasing is the main source of intuition for the proof, so we think it is worth presenting it.

We denote $v_1,\dots,v_p$ the columns of $V$. If we multiply $V$ by a suitable orthogonal $p\times p$ matrix (which does not change the fact that $V$ is second-order critical for~\eqref{opt:MCLF}), we can assume that $\scal{v_1}{\un_n}\geq 0$ and $\scal{v_k}{\un_n}=0$ for all $k\geq 2$. \footnote{If $V$ does not satisfy this condition, consider instead $VG$ where $G$ is an orthogonal matrix such that all but its first columns are orthogonal to $V^T \un_n$ and $\scal{G_{:1}}{V^T \un} > 0$.} This is summarized by the following assumption.

\begin{assumption*}\label{ass1}
  For $i=2,\dots,p$, $(V^T\un_n)_i = \scal{v_i}{\un_n} = 0$. Moreover, $\scal{v_1}{\un_n} \ge 0$. In particular,
  $$\norm{V^T \un_n} = \scal{v_1}{\un_n} \le \norm{v_1}\norm{\un_n} \le n.$$
\end{assumption*}
Showing that the condition number satisfies $\frac{\lambda_n(\bm{L})}{\lambda_2(\bm{L})}\ge p$ can be recast as showing that the value of the following optimization problem is at least $p$:

\begin{equation}\label{P1}
\begin{aligned}
\underset{(L,\mu) \in \s^{n \times n} \times \Rn}{\mbox{inf}} \quad & \frac{\lambda_n(L)}{\lambda_2(L)}\\
\textrm{s.t.} \quad & L \succeq 0,\\
& L\un_n = 0, \\
& \lambda_2(L) > 0, \\
& (L-\diag(\mu))V = 0, \\
&  \scal{(L-\diag(\mu))\dot{V}}{\dot{V}} \ge 0 \mbox{ for all } \dot{V} \in T_V(\s^{p-1})^n.
\end{aligned}
\end{equation}
Indeed, the point $(\bm{L},\hat{\mu})$, with $\hat{\mu} = \diag(\bm{L}VV^T)$ is feasible for the above problem since $V$ is a second-order point of~\eqref{opt:MCLF}. Hence, the condition number of $\bm{L}$ is greater than or equal to the optimal value of problem~\eqref{P1}.

A similar approach which consists in recasting the problem as finding the worse condition number among possible cost matrices is found in \citep*{zhang2024improved} where the author considers semidefinite programs and their Burer-Monteiro factorization without affine constraints.

Note that the first three constraints represent the fact that the semidefinite relaxation admits a rank $1$ solution and the last two, the first and second-order optimality conditions on $V$. Problem~\eqref{P1} is not convex, but we will see that it has the same optimal value as a convex problem.

First, define the set $K$ as the smallest convex cone containing $\{\dot{V}\dot{V}^T, \dot{V}\in T_V(\s^{p-1})^n \}$. The last constraint of problem~\eqref{P1} is equivalent to $\diag(\mu)-L \in K^\circ$, where $K^\circ$ is the polar cone of $K$ :
\begin{equation*}
    K^\circ = \{M\in\s^{n\times n}:\scal{M}{N}\le 0 \mbox{ for all } N \in K \}.
\end{equation*}
Now, consider the following convex minimization problem:
\begin{align}\label{primal}\tag{Primal}
\underset{(L,\mu) \in \s^{n \times n} \times \Rn}{\mbox{inf}} \quad & \lambda_n(P_{\perp}LP_{\perp})\\
\textrm{s.t.} \quad & P_{\perp}LP_{\perp} \succeq P_{\perp}, \nonumber\\
& (P_{\perp}LP_{\perp}-\diag(\mu))V = 0,  \nonumber\\
&  \diag(\mu)-P_{\perp}LP_{\perp} \in K^\circ,  \nonumber
\end{align}
where $P_{\perp}=I_n-n^{-1}\un_n\un_n^T$ is the projection matrix on the space orthogonal to $\un_n$. We have the following two lemmas, whose proofs can be found in~\ref{proof_sec_4}.
\begin{lem}\label{lemmeP1primal}
    Problem~\eqref{P1} and problem~\eqref{primal} have the same optimal value.
\end{lem}
In order to find a lower bound on the convex problem~\eqref{primal}, a natural idea is to introduce the dual problem and find an appropriate dual certificate.
\begin{lem}\label{lemdual}
    The dual problem of~\eqref{primal} is
\begin{align}\label{opt:dual}\tag{Dual}
\underset{(W,Z,H)\in\Rnp \times \s^{n \times n} \times \s^{n \times n}}{\sup} \quad & \left\langle Z, P_{\perp} \right\rangle  \nonumber\\
\mathrm{s.t.} \quad & Z \succeq 0, \nonumber\\
  & H \in K,  \nonumber\\
  & \diag(WV^T) = \diag(H),  \nonumber\\
  & M = P_{\perp}\left(Z + H -\frac{1}{2}\left(WV^T + VW^T\right)\right)P_{\perp},  \nonumber\\
  & M \succeq 0,  \nonumber\\
  & \tr(M) \le 1. \nonumber
\end{align}
\end{lem}
By duality, to show that the optimal value of~\eqref{primal} is at least $p$, it suffices to exhibit a dual certificate, i.e. a point feasible for~\eqref{opt:dual} for which $\scal{Z}{P_{\perp}}\geq p$. A similar duality argument is used in \citep*{zhang2024improved}, to lower bound the condition number of the objective function.

\subsection{Remarks on a proof without using the dual\label{ss:remarks}}

\subsubsection{Possible proof without explicit dual formulation}

It is important to note that it is possible to provide a lower bound on the condition number of $\bm{L}$ without passing to the dual problem, provided that we have found an appropriate choice of matrices $(W,Z,H)$ that satisfy the constraints of problem \eqref{opt:dual}, except for the trace constraint on $M=P_\perp(Z+H - \frac{1}{2}(WV^T+VW^T))P_\perp$. Indeed, since $M \succeq 0$, to make the largest eigenvalue of $\bm{L}$ appear, we can upper bound $\scal{\bm{L}}{M}$ as follows 
\begin{equation}\label{ineqrmk1}
  \scal{\bm{L}}{M} \le \lambda_n(\bm{L})\tr(M).
\end{equation}
We are now going to lower bound $\scal{\bm{L}}{M}$ to make $\lambda_2(\bm{L})$ appear:
\begin{align}
    \scal{\bm{L}}{M} & = \scal{\bm{L}}{P_\perp (Z + H -\frac{1}{2}(WV^T+VW^T)) P_\perp} \nonumber \\
    & = \scal{\bm{L}}{P_\perp Z} + \scal{\bm{L}}{H} -\scal{\bm{L}}{WV^T} \nonumber \\
                     & = \scal{\bm{L}}{P_\perp Z} + \scal{\bm{L}}{H} -\underbrace{\scal{\bm{L} - \ddiag(\bm{L}VV^T)}{WV^T}}_{\mbox{\footnotesize $= 0$ since \eqref{FOCP} holds}} \nonumber\\
  & \textcolor{white}{=} -\scal{\ddiag(\bm{L}VV^T)}{WV^T} \nonumber\\
    & = \scal{\bm{L}}{P_\perp Z} + \scal{\bm{L}}{H} - \underbrace{\scal{\ddiag(\bm{L}VV^T)}{H}}_{\mbox{\footnotesize since $\diag(H) = \diag(WV^T)$}} \label{eq:use_diag_constraint} \\
    & = \scal{\bm{L}}{P_\perp Z} + \underbrace{\scal{\bm{L} - \ddiag(\bm{L} VV^T)}{{H}}}_{\mbox{\footnotesize $\ge 0$ since \eqref{socp3} \mbox{holds}} }  \label{eq:use_2nd_order} \\
    & \ge \lambda_2(\bm{L}) \tr(P_\perp Z). \label{eq:lb_lambda2}
\end{align}
The second equality comes from the fact that $P_\perp \bm{L} = \bm{L}$ because $\bm{L} \un_n = 0$, and the last one from the fact that $\bm{L}\succeq \lambda_2(\bm{L}) P_{\perp}$ as all eigenvectors which are not associated with the eigenvalue $0$ are orthogonal to $\un_n$. Combining the two bounds on $\scal{\bm{L}}{M}$, we get, if $\tr(M) > 0$,
\begin{equation}\label{inegcond}
    \lambda_2(\bm{L})\tr (P_\perp Z) \le \lambda_n(\bm{L}) \tr(M) \implies \frac{\tr(P_{\perp}Z)}{\tr(M)} \le \frac{\lambda_n(\bm{L})}{\lambda_2(\bm{L})}.
\end{equation}
Thus, finding an appropriate $(W,Z,H)$ such that $\tr(P_\perp Z) \ge p\tr(M)$ proves the theorem. The aforementioned remark is, of course, just a rephrasing of weak duality without explicitly computing the dual, but a similar approach was used in \citep*{ling2023local} to derive the main theorem without passing to the dual formulation.
\subsubsection{Rephrasing of the proof of \citep*[Theorem 2.1]{ling2023local} in terms of dual certificate}
We can reinterpret the proof of the \citep*[Theorem 2.1]{ling2023local} for $d=1$ as finding an appropriate triplet $(W,Z,H)$ and following the reasoning described in the previous paragraph. Indeed, in the proof, the author makes the implicit choice
\begin{align*}
    W & = (p-1) V, \\
    Z & = (p-1) VV^T, \\
    H & = (p-2)\un_n\un_n^T + (VV^T)^{\odot2}.
\end{align*}
We have $Z \succeq 0$ and $H_{ii} = (WV^T)_{ii} = (p-1)$ for all $i$. Moreover, the author shows that $H = \mathbb{E}(\dot{V}\dot{V}^T)$ with a careful choice of random tangent vectors $\dot{V}$ (see subsection~\ref{subsectionconstraints} for more detail); in other words, $H$ is an average of elements of $K$. \newline
To further parallel our proof to that of Ling, with the above choice of dual certificate, we note that Equation~\eqref{ineqrmk1} corresponds to the Ling's equation (3.6), Equation~\eqref{eq:lb_lambda2} to (3.7) and line~\eqref{eq:use_diag_constraint} to the equation between (3.6) and (3.7). \newline
With these choices of dual certificates, one has
\begin{align*}
    M & = P_\perp (Z + H - \frac{1}{2}(WV^T + VW^T)) P_\perp \\
    & = P_\perp((p-2)\un_n\un_n^T + (VV^T)^{\odot 2} + (p-1)VV^T -(p-1)VV^T)P_\perp \\
    & = P_\perp(VV^T)^{\odot 2}P_\perp \\
    & \succeq 0.
\end{align*}
Then, thanks to \eqref{inegcond}, one gets
\begin{equation*}
    \frac{\lambda_n(\bm{L})}{\lambda_2(\bm{L})} \ge \frac{\tr(P_\perp Z)}{\tr(M)} = (p-1)\frac{\tr(P_\perp VV^T)}{\tr(P_\perp (VV^T)^{\odot 2} P_\perp )}\footnote{Note that $\tr(P_\perp (VV^T)^{\odot 2}P_\perp) > 0$. If it were not the case then $V$ would be optimal.}.
\end{equation*}
We have on the one hand
\begin{align*}
    \tr(P_\perp VV^T) & = \tr(VV^T) - n^{-1} \scal{VV^T}{\un_n \un_n} \\
    & \ge \frac{n-n^{-1}\sum_{ij}\scal{v_i}{v_j}^2}{2} \mbox{ using \citep*[Lemma~3.2]{ling2023local}}.
\end{align*}
On the other hand
\begin{align*}
    \tr(P_\perp (VV^T)^{\odot 2} P_\perp) & = \tr((VV^T)^{\odot 2}) - n^{-1}\scal{\un_n \un_n}{(VV^T)^{\odot 2}} \\
    & = n - n^{-1}\sum_{ij}\scal{v_i}{v_j}^2.
\end{align*}
Combining the three previous inequalities, we get the bound proved by the author
\begin{equation*}
    \frac{\lambda_n(\bm{L})}{\lambda_2(\bm{L})} \ge \frac{p-1}{2}.
  \end{equation*}
The main difference between our proof and that of \citep*{ling2023local}, other than the primal-dual formulation, is a more careful definition of $W$; our choices of $Z$ and $H$ are the same up to multiplicative constants.
\subsection{Choice of the dual certificate}\label{subsectionconstraints}
In the following, we aim to find an adequate dual certificate $(W_*,Z_*,H_*)$. A natural choice for $Z_*$ (which ensures $\scal{Z_*}{P_{\perp}}=p$) is
\begin{equation*}
    Z_* \defeg \frac{p}{\scal{P_\perp}{VV^T}} VV^T.
\end{equation*}
Note that since $VV^T$ is not colinear to $\un_n\un_n^T$, $\scal{P_\perp}{VV^T} > 0$ so $Z_*$ is well defined.

For $H_*$, we choose
\begin{equation}\label{Hval}
    H_* \defeg \beta((p-2)\un_n \un_n^T + (VV^T)^{\odot{2}}) \mbox{ for some } \beta \ge 0,
\end{equation}
which belongs to $K$. Indeed, \citep*{ling2023local} and \citep*{mcrae2024benign} showed that $H_* = \beta \mathbb{E}(\dot{V}\dot{V}^T)$ with $\dot{V} = P_{T_V}(\un_n \Phi^T)$ and $\Phi \sim \mathcal{N}(0,I_p)$. Since $\dot{V}\dot{V}^T \in K$, by convexity and closure of $K$, it holds that $\mathbb{E}(\dot{V}\dot{V}^T) \in K$.

There is no straightforward choice for $W_*$. A natural one would be $W_*=\beta(p-1)V$ as it would satisfy the diagonal constraint:
\begin{align*}
  \diag(W_*V^T)
  & = \beta(p-1)\un_n \\
  & = \beta\left((p-2)\un_n + \diag\left((VV^T)^{\odot 2}\right)\right) \\
  & =\diag(H_*).
\end{align*}
However, numerical experiments suggest that it does not work. Fortunately, this can be corrected by adding to $\beta(p-1)V$ a matrix proportional to $W_*'=\un_n\un_n^TV + \varepsilon$, for some $\varepsilon\in\R^{n\times p}$ chosen so that $\diag(W_*'V^T)=0$.
We choose $\varepsilon=-\diag(VV^T\un_n)V$ (so that $W'_*=P_{T_V}(\un_n\un_n^TV)$). Under assumption~\ref{ass1},
\begin{equation*}
W'_* = \scal{v_1}{\un_n}\left(\un_ne_1 - \diag(v_1)V\right),
\end{equation*}
which suggests the choice
\begin{equation}\label{Wval}
    W_* = \beta(p-1)V + \delta (\un_ne_1 - \diag(v_1)V)  \mbox{ for some } \delta \in \R.
\end{equation}
\subsection{Constraints}
The goal now is to find $\beta, \delta$ such that the dual certificate $(W_*,Z_*,H_*)$ defined in the previous subsection satisfies the constraints of problem~\eqref{opt:dual} and $\scal{Z_*}{P_{\perp}} \ge p$.

The definition of $Z_*$ immediately implies
\begin{equation*}
  \scal{Z_*}{P_{\perp}} = p \mbox{ and } Z_* \succeq 0.
\end{equation*}
Furthermore, $H_*$ defined in~\eqref{Hval} is in $K$ if $\beta \ge 0$, and the definition of $W_*$ ensures that the equality $\diag(W_*V^T) = \diag(H_*)$ holds true. Therefore, we only have to find $\beta, \delta$ such that
\begin{subequations}
\begin{align}
\label{cont1}
  \beta & \ge 0 \\ \label{cont2}
  M_* &\succeq 0 \\ \label{cont4}
  \tr(M_*)&\leq 1,
\end{align}
\end{subequations}
where of course $M_* = P_{\perp}(Z_* + H_* -\frac{1}{2}(W_*V^T + VW_*^T))P_{\perp}$.

For the positive semidefiniteness of $M_*$, we have the following lemma, proved in~\ref{proof_sec_4}.

\begin{lem}\label{lemb}
    Under assumption~\ref{ass1}, if $\beta\geq 0$, $M_*$ is positive semidefinite if
    \begin{equation}
      \label{H2}
      \left(\frac{p}{2(p-1)\scal{P_{\perp}}{VV^T}}\right)^2
               \ge \left(\beta - \frac{p}{2(p-1)\scal{P_{\perp}}{VV^T}}\right)^2
                + \left(\frac{\delta}{2 \sqrt{p-1}} \right)^2.
    \end{equation}
\end{lem}
The idea of the proof is to write $M_* = R (I_p\otimes S) R^T$ for an appropriate $R\in\R^{n\times 2p}$ and
\begin{equation*}
    S = \begin{pmatrix}
      \frac{p}{\scal{P_\perp}{VV^T}}-(p-1)\beta & \frac{\delta}{2} \\
      \frac{\delta}{2} & \beta \\
    \end{pmatrix}.
  \end{equation*}
It is straightforward to see that $M_* \succeq 0$ if $S \succeq 0$ and the latter is verified under condition \eqref{H2}.
Now, the trace of $M_*$ is given by
\begin{align*}
    \tr (M_*) & = \tr (P_{\perp}Z_*) + \tr (P_{\perp}H_*) - \tr (P_{\perp}W_*V^T) \\
              & = p + \beta\scal{P_\perp}{(VV^T)^{\odot 2}}
                - \tr (P_{\perp}(\beta(p-1)VV^T -\delta \diag(v_1) VV^T)) \\
              & = \beta\scal{P_\perp}{(VV^T)^{\odot 2}-(p-1)VV^T}
                + \delta \scal{P_{\perp}}{\diag(v_1)VV^T} + p.
\end{align*}
We must therefore find $\beta\geq 0$ and $\delta$ satisfying Equation~\eqref{H2} such that
\begin{gather}
  t_1\beta
  + \frac{t_2 \delta}{2\sqrt{p-1}} \geq p-1, \label{H3}\\
  \mbox{where }t_1 = \scal{P_\perp}{(p-1)VV^T - (VV^T)^{\odot 2}} \nonumber \\
  \mbox{ and }t_2 = - 2\sqrt{p-1} \scal{P_{\perp}}{\diag(v_1)VV^T}.\nonumber
\end{gather}
We set
\begin{align*}
  \beta & = \frac{p}{2(p-1)\scal{P_{\perp}}{VV^T}}
          \left(1 + \frac{t_1}{\sqrt{t_1^2+t_2^2}}\right), \\
  \delta & = \frac{p}{\sqrt{p-1}\scal{P_{\perp}}{VV^T}} \frac{t_2}{\sqrt{t_1^2+t_2^2}}.
\end{align*}
With this definition, Equation~\eqref{H2} is true, as a direct computation shows. It remains to show that Equation~\eqref{H3} is also true, which is equivalent to
\begin{align*}
  \sqrt{t_1^2+t_2^2}
  \geq \frac{2(p-1)^2 \scal{P_{\perp}}{VV^T}}{p} - t_1,
\end{align*}
and therefore implied by $t_1^2+t_2^2 \geq
\left(\frac{2(p-1)^2 \scal{P_{\perp}}{VV^T}}{p} - t_1\right)^2$. This last inequality is equivalent, from the definitions of $t_1,t_2$, to
\begin{align*}
    0 & \le t_2^2 - \frac{4(p-1)^4}{p^2}\scal{P_\perp}{VV^T}^2 + \frac{4(p-1)^2}{p}\scal{P_\perp}{VV^T}t_1 \\
    & = 4(p-1)\scal{P_\perp}{\diag(v_1)VV^T}^2 - \frac{4(p-1)^4}{p^2}\scal{P_\perp}{VV^T}^2 \\
    &\textcolor{white}{=} + \frac{4(p-1)^3}{p}\scal{P_\perp}{VV^T}^2 - \frac{4(p-1)^2}{p}\scal{P_\perp}{VV^T}\scal{P_\perp}{(VV^T)^{\odot 2}} \\
    & = 4(p-1)\big(\scal{P_\perp}{\diag(v_1)VV^T}^2 + \frac{(p-1)^2}{p^2}\scal{P_\perp}{VV^T}^2 \\
    & \textcolor{white}{=} - \frac{p-1}{p}\scal{P_\perp}{VV^T}\scal{P_\perp}{(VV^T)^{\odot 2}}\big),
\end{align*}
which simplifies to :
\begin{align}
  \scal{P_{\perp}}{\diag(v_1)VV^T}^2
  & + \frac{(p-1)^2}{p^2}\scal{P_{\perp}}{VV^T}^2 \nonumber \\
  & - \frac{p-1}{p}\scal{P_{\perp}}{VV^T}\scal{P_\perp}{(VV^T)^{\odot 2}} \geq 0.
    \label{eq:final_lemma_44}
\end{align}
We observe that
\begin{align}
  \scal{P_{\perp}}{(VV^T)^{\odot 2}}
  & = \scal{I_n - \frac{1}{n}\un_n\un_n^T}{(VV^T)^{\odot 2}} \nonumber \\
  & = n - \frac{||VV^T||_F^2}{n} \label{eq:norm_VV2_a} \\
  & = n - \frac{||V^TV||_F^2}{n} \nonumber \\
  & = n - \frac{\sum_{i,j=1}^p\scal{v_i}{v_j}^2}{n} \nonumber \\
  & \leq n - \frac{\sum_{i=1}^p ||v_i||^4}{n} \nonumber \\
  & = n - \frac{||v_1||^4}{n} -  \frac{\sum_{i=2}^p ||v_i||^4}{n} \nonumber \\
  & \leq n - \frac{||v_1||^4}{n} -  \frac{\left(\sum_{i=2}^p ||v_i||^2\right)^2}{n(p-1)}
  \label{eq:norm_VV2_b}\\
  & = n - \frac{||v_1||^4}{n} -  \frac{\left(n-||v_1||^2\right)^2}{n(p-1)}.
    \label{eq:norm_VV2_c}
\end{align}
At line~\eqref{eq:norm_VV2_a}, we used $\diag(VV^T)=\un_n$. At line~\eqref{eq:norm_VV2_b}, we used Cauchy-Schwarz and, at line~\eqref{eq:norm_VV2_c}, we used that $\sum_{i=1}^p ||v_i||^2 = ||V||_F^2 = \operatorname{Tr}(\diag(VV^T))=n$.

In addition, from assumption~\ref{ass1},
\begin{align*}
  \scal{P_{\perp}}{VV^T}
  & = \scal{I_n-\frac{1}{n}\un_n\un_n^T}{VV^T} \\
  & = n - \frac{\scal{v_1}{\un_n}^2}{n},
\end{align*}
and
\begin{align*}
  \scal{P_{\perp}}{\diag(v_1)VV^T}
  & = \scal{\diag(v_1)}{VV^T} - \frac{1}{n}\scal{\un_n\un_n^T}{\diag(v_1)VV^T} \\
  & = \scal{v_1}{\diag(VV^T)} - \frac{1}{n} \underbrace{\scal{v_1}{VV^T\un_n}}_{
    =\sum_{k=1}^p \scal{v_1}{v_kv_k^T\un_n}}\\
  & = \scal{v_1}{\un_n}\left(1 - \frac{||v_1||^2}{n}\right)
    \mbox{ as }v_k^T\un_n=0\mbox{ for }k\geq 2.
\end{align*}

We combine the last three equations. They show that the left-hand side of Equation~\eqref{eq:final_lemma_44} is lower bounded by
\begin{align*}
    \scal{P_{\perp}}{\diag(v_1)VV^T}^2 & + \frac{(p-1)^2}{p^2}\scal{P_{\perp}}{VV^T}^2 \\
   & \textcolor{white}{ = } - \frac{p-1}{p}\scal{P_{\perp}}{VV^T}\scal{P_\perp}{(VV^T)^{\odot 2}} \\
  & \ge  \scal{v_1}{\un_n}^2\left(1 - \frac{||v_1||^2}{n}\right)^2
    + \frac{(p-1)^2}{p^2}\left(n-\frac{\scal{v_1}{\un_n}^2}{n}\right)^2 \\
  & - \frac{p-1}{p} \left(n-\frac{\scal{v_1}{\un_n}^2}{n}\right)
    \left(n-\frac{||v_1||^4}{n}-\frac{(n-||v_1||^2)^2}{n(p-1)}\right) \\
  & = ||v_1||^4
    - \frac{2}{p} \left( \frac{p-1}{n}\scal{v_1}{\un_n}^2
    + n\right) ||v_1||^2 \\
  & \qquad + \frac{1}{p^2}\left(\frac{(p-1)^2}{n^2}\scal{v_1}{\un_n}^4
    + 2(p-1) \scal{v_1}{\un_n}^2
    + n^2\right) \\
  & = \left(||v_1||^2 - \frac{1}{p}\left(n+\frac{p-1}{n}\scal{v_1}{\un_n}^2\right) \right)^2 \\
  & \geq 0.
\end{align*}
Equation~\eqref{eq:final_lemma_44} is therefore true, which concludes.

\paragraph{Acknowledgements}

The authors would like to thank Antonin Chambolle for helpful discussions and feedback for this work, partially funded by ANR-23-PEIA-0004 (PDE-AI), as well as Andrew McRae. Faniriana Rakoto Endor and Irène Waldspurger have been supported by the French government under management of Agence Nationale de la Recherche as part of the “Investissements d’avenir” program, reference ANR19-P3IA-0001 (PRAIRIE 3IA Institute).

\appendix
\section{Technical lemmas}
\label{Appendix}
\subsection{Proofs of section~\ref{section2}}\label{proof_sec_2}
\begin{proof}[Proof of proposition~\ref{propopti}]
  Let us set $\bm{x}=\un_n$ and let $V\in(\s^{p-1})^n$ be such that
\begin{align}
    V^T\un_n & = 0, \label{contV_1}\\
    V^TV & = \frac{n}{p}I_p, \label{contV_2}\\
    \scal{v_i \odot v_j \odot v_k}{\un_n} & = 0, \mbox{ for all } 1 \le i,j,k \le p, \label{contV_3}
\end{align}
where the $v_l$'s are the columns of $V$. Such matrices $V$ exist at least when $p$ is even or $n$ is; an example is provided at the end of the proof.
Now, set
\begin{equation*}
C = -(P_V+pP_{V^{\perp}}-pP_{\un}),
\end{equation*}
where $P_V=\frac{p}{n}VV^T$ is the orthogonal projector onto $\operatorname{Range}(V)$, $P_{V^{\perp}}=I_n-\frac{p}{n}VV^T$ the projector onto $\operatorname{Range}(V)^{\perp}$ and $P_{\un}= \frac{1}{n}\un_n\un_n^T$ the projector onto $\R\un_n$.

Since $C\un_n=0$, the Laplacian matrix is $\bm{L}=-C=P_V+pP_{V^{\perp}}-pP_{\un}$. Its eigenvalues are $0$ (with eigenspace $\R\un_n$), $1$ (with eigenspace $\operatorname{Range}(V)$) and $p$ (with eigenspace $(\operatorname{Range}(V)\oplus \R\un_n)^{\perp}$). Therefore, $\bm{L}\succeq 0,\lambda_2(\bm{L})>0$ and its condition number is $p$.

    Using~\eqref{FOCP} and~\eqref{eq_hess}, $V$ is second-order optimal if
    \begin{align}
      SV&=0, \nonumber \\
      \scal{\hess_V(\dot{V})}{\dot{V}} =
        2\scal{S\dot{V}}{\dot{V}} & \ge 0, \forall~\dot{V} \in T_V(\s^{p-1})^{n} \label{eq_SVV},
    \end{align}
    where
    \begin{equation*}
        S\defeg\bm{L}-\ddiag(\bm{L}VV^T) =\bm{L}-I_n = (p-1)P_{V^\perp} - p P_{\un_n}.
    \end{equation*}
    It is clear that with this choice of $C$, $\bm{L}V=V$, hence $SV=0$. It remains to show that the Hessian is positive semidefinite at $V$. The difficulty stems from the fact that $S$ has a negative eigenvalue: $S\un_n=-\un_n$.
We first exhibit a subspace of $T_V(\s^{p-1})^n$ included in $\operatorname{Ker}\hess_V$. Then, we prove Equation~\eqref{eq_SVV} by decomposing $\dot{V}$ onto the kernel and its orthogonal.

Note that any matrix of the form
\begin{equation}\label{kerhess}
  \diag(Va)V-\un_n a^T,
\end{equation}
with $a\in \Rp$, belongs to $T_V(\s^{p-1})^n$ and to the kernel of $\hess_V$. Indeed, for any $a\in\Rp$,
\begin{align*}
  \frac{1}{2}
  \hess_V\left( \diag(Va)V-\un_n a^T \right)
  & = P_{T_V}\left( S(\diag(Va)V-\un_n a^T) \right) \\
  & = P_{T_V} \left( -\frac{p(p-1)}{n}VV^T\diag(Va)V \right) \\
  & + P_{T_V} ( \underbrace{(p-1)\diag(Va)V}_{ \in (T_V(\s^{p-1})^n)^\perp } )\\
  & -P_{T_V}\left((p-1)\un_na^T \right) \\
  & - P_{T_V} \left(\frac{p}{n}\un_n\un_n^T\diag(Va)V-p\un_na^T \right) \\
  & = P_{T_V} \left( -\frac{p(p-1)}{n}VV^T\diag(Va)V\right) \\
  & + P_{T_V}\Big(\un_na^T-\frac{p}{n} \un_na^T\underbrace{V^TV}_{=\frac{n}{p}I_n} \Big) \\
  & = -\left(\frac{p(p-1)}{n}\right) P_{T_V}\left( VV^T\diag(Va)V \right).
\end{align*}
For $1 \le j,k \le p$, we have
\begin{equation*}
  (V^T\diag(Va)V)_{jk} = \sum_{i=1}^p a_i \underbrace{\scal{v_i \odot v_j \odot v_k}{\un_n}}_{=0} = 0,
\end{equation*}
hence $V^T\diag(Va)V=0$, and $\hess_V  \left( \diag(Va)V-\un_n a^T \right) =0$.
    
Let us fix $\dot{V}\in T_V(\s^{p-1})^n$. It can be decomposed as $\dot{V}=X+Y$, for some $X,Y \in T_V(\s^{p-1})^n$ such that $X\in \ker{\hess_V}$ and $Y\in (\ker{\hess_V})^\perp$. Since $Y$ is orthogonal to the kernel of $\hess_V$, it is orthogonal to any matrix of the form~\eqref{kerhess}. Therefore, for any $a \in \Rp$,
    \begin{align*}
        0 & =\scal{\diag(Va)V-\un_na^T}{Y}&\\
        & = \scal{Va}{\diag(YV^T)}-\scal{\un_na^T}{Y}& \\
        & = -\scal{\un_na^T}{Y} &\mbox{(since }Y\in T_V(\s^{p-1})^n)\\
        & = -\scal{a^T}{\un_n^TY},&
    \end{align*}
    which implies that $\un_n^TY=0$. Hence, it holds that
    \begin{equation*}
        SY = (P_V+pP_{V^\perp}-I_n)Y - \frac{p}{n}\un_n\un_n^TY = (P_V+pP_{V^\perp}-I_n)Y.
    \end{equation*}
    Finally,
    \begin{align*}
        \scal{\hess_V(\dot{V})}{\dot{V}} & = \overbrace{\scal{\hess_V(X)}{X}}^{=0} + \overbrace{2\scal{\hess_V(X)}{Y}}^{=0} + \overbrace{\scal{\hess_V(Y)}{Y}}^{2\scal{SY}{Y}} \\
                                         & = 2\scal{(P_V+pP_{V^\perp}-I_n)Y}{Y} \\
                                         & = 2(p-1)\scal{P_{V^{\perp}}Y}{Y} \\
                                         & \geq 0.
    \end{align*}

    To conclude, we show the existence of $V$ satisfying equations~\eqref{contV_1},~\eqref{contV_2} and~\eqref{contV_3}.
    For instance, when $p$ is even, for any $j \in\left\{1,\dots \frac{p}{2}\right\}$ and $i\in\{1,\dots,n\}$, we can set
\begin{equation*}
    V_{i,2j-1} = \sqrt{\frac{2}{p}}\cos\left(\frac{2 \pi m_j}{n} (i-1)\right) \mbox{ and }
    V_{i,2j} = \sqrt{\frac{2}{p}}\sin\left(\frac{2 \pi m_j}{n} (i-1)\right),
  \end{equation*}
  where $m_j=3j-2$. All three equations can be proved using similar computations. Let us for instance establish equality~\eqref{contV_3} in the case where $i,j,k$ are odd. We have

\begin{align*}
  \scal{v_i \odot v_j \odot v_k}{\un_n}
  & = \left( \frac{2}{p}\right)^{\frac{3}{2}} \sum_{l=0}^{n-1} \cos\left(\frac{2 \pi m_i}{n}l\right) \cos\left(\frac{2 \pi m_j}{n}l\right) \cos\left(\frac{2 \pi m_k}{n}l\right) \\
  & = \frac{1}{\sqrt{2p^3}} \sum_{l=0}^{n-1} \cos\left(\frac{2 \pi}{n}(m_i+m_j+m_k)l\right) \\
  & ~~~~~~~~~~~~~~~~~~~~~+ \cos\left(\frac{2 \pi}{n}(m_i+m_j-m_k)l\right) \\
  & ~~~~~~~~~~~~~~~~~~~~~+ \cos\left(\frac{2 \pi}{n}(m_i-m_j+m_k)l\right) \\
  & ~~~~~~~~~~~~~~~~~~~~~+ \cos\left(\frac{2 \pi}{n}(m_i-m_j-m_k)l\right).
\end{align*}
This sum is zero because one can check that, for any $\varepsilon_j,\varepsilon_k\in\{\pm 1\}$, $m_i+\varepsilon_j m_j +\varepsilon_k m_k \not\equiv 0[n]$.

If $p$ is odd but $n$ is even, we can make the same construction for the first $p-1$ columns of $V$ and add one last column whose entries alternate between $-\sqrt{\frac{1}{p}}$ and $\sqrt{\frac{1}{p}}$.

\end{proof}
\subsection{Proofs of section~\ref{section3}}\label{proof_sec_3}
\begin{proof}[Proof of Corollary~\ref{cor1}]
    The Laplacian matrix of $C$ defined in~\eqref{costZ2synch1} is
    \begin{equation*}
        \bm{L} = n(I_n-n^{-1}\bm{xx}^T) + \sigma(\ddiag( W \bm{xx}^T) - W).
    \end{equation*}
    Define the following matrix:
    \begin{equation*}
        \bm{L}^{ W} = \ddiag( W \bm{xx}^T) - W.
    \end{equation*}
    Since $I_n-n^{-1}\bm{xx}^T$ is the orthogonal projector on the orthogonal space of $\bm{x}$, its eigenvalues are $0$ (with multiplicity $1$) and $1$ (with multiplicity $n-1$).

    Therefore, using Weyl's inequality,
    \begin{align*}
        \lambda_n(\bm{L}) & \le n + \sigma \norm{\bm{L}^{ W}}, \\
        \lambda_2(\bm{L}) & \ge n - \sigma \norm{\bm{L}^{ W}}.
    \end{align*}
    We need to upper bound $\norm{\bm{L}^{ W}}$. The triangular inequality gives
    \begin{align*}
        \norm{\bm{L}^{ W}} \le \norm{ W \bm{x}}_{\infty} + \norm{ W}.
    \end{align*}
    Moreover, for all $\varepsilon' > 0$, it holds that
    \begin{equation}\label{eq:Wx_inf}
        \norm{ W \bm{x}}_{\infty} \le \sqrt{(2+\varepsilon')n \log n},
    \end{equation}
    with probability at least $1-n^{-\varepsilon'/2}$. Indeed, note that, for all $i\leq n$, $(W\bm{x})_i \sim \mathcal{N}(0,n-1)$. Therefore, from \citep*[Prop 2.1.2]{vershynin2018high}, for all $t>0$,
      \begin{equation*}
        \mathbb{P}(\lvert(W\bm{x})_i\rvert > t) \leq \sqrt{\frac{2(n-1)}{\pi}} \frac{e^{-\frac{t^2}{2(n-1)}}}{t}.
      \end{equation*}
    Applying a union bound and taking $t=\sqrt{(2+\varepsilon')n \log n}$ yields~\eqref{eq:Wx_inf}. Moreover, it is also true that, with probability at least $1-4e^{-n}$,
    \begin{equation*}
        \norm{W} \leq c_0 \sqrt{n},
    \end{equation*}
    for some universal constant $c_0>0$. This is an immediate consequence of \citep*[Corollary 4.3.6]{vershynin2018high}. Therefore, for any $\varepsilon>0$, it holds with probability at least $1-n^{-\varepsilon/4}-4e^{-n}$ that
    \begin{align*}
      \norm{\bm{L}^W}
      & \le \sqrt{\left(2+\frac{\varepsilon}{2}\right)n \log n} + c_0 \sqrt{n}
      \\
      & \le \sqrt{\left(2+\varepsilon\right)n \log n}
      \hskip 1cm\mbox{(for }n\mbox{ large enough).}
    \end{align*}
    Then, with the same probability,
    \begin{align*}
      \sigma <\frac{p-1}{p+1}\sqrt{\frac{n}{(2+\varepsilon)\log n}}
      & \iff \frac{n + \sigma\sqrt{(2+\varepsilon)n \log n} }{
        n-\sigma\sqrt{(2+\varepsilon)n \log n}}<p
         \\
        & \implies \frac{\lambda_n(\bm{L})}{\lambda_2(\bm{L})} < p.
    \end{align*}
    Furthermore, since $\lambda_2(\bm{L}) \ge n - \sigma \sqrt{(2+\varepsilon)n \log n}$, it is true that $\lambda_2(\bm{L}) > 0$  and $\bm{L} \succeq 0$ for $n$ large enough, with probability at least $1-n^{-\varepsilon/4}-4e^{-n}$. The conclusion follows from theorem~\ref{thm1}.
\end{proof}

\begin{proof}[Proof of corollary~\ref{Z2ber}]
  Let $\varepsilon>0$ be fixed.
    We can assume without loss of generality that the vector we want to reconstruct is $\bm{x}=\un_n$ (see section~\ref{section4} for more details). The Laplacian matrix is
    \begin{equation*}
        \bm{L} =\diag(C\un_n)-C.
    \end{equation*}
    Note that $\bm{L}$ can be decomposed as a principal term and a noise term as follows:
\begin{align*}
    \bm{L} & = \mathbb{E}(\bm{L})+(\bm{L}-\mathbb{E}(\bm{L})) \\
    & = \underbrace{\delta(nI_n-\un_n\un_n^T)}_{\mbox{principal term}} + \underbrace{(\bm{L}-\mathbb{E}(\bm{L}))}_{ \mbox{noise term}}.
\end{align*}
Therefore, using Weyl's inequality yields
\begin{align*}
    \lambda_n(\bm{L}) & \le \delta n + \norm{\bm{L}-\mathbb{E}(\bm{L})}, \\
    \lambda_2(\bm{L}) & \ge \delta n - \norm{\bm{L}-\mathbb{E}(\bm{L})}.
\end{align*}
In particular, as soon as $\norm{\bm{L}-\mathbb{E}(\bm{L})}<\delta n$, $\lambda_2(\bm{L})>0$ and
\begin{equation*}
\frac{\lambda_n(\bm{L})}{\lambda_2(\bm{L})}
\leq \frac{\delta n+ \norm{\bm{L}-\mathbb{E}(\bm{L})}}{\delta n- \norm{\bm{L}-\mathbb{E}(\bm{L})}},
\end{equation*}
so that, from theorem~\ref{thm1}, all second-order critical points are global minimizers if the right-hand side of the above is below $p$.

Note that
\begin{align*}
    \norm{\bm{L}-\mathbb{E}(\bm{L})} \le \norm{\diag(C\un_n)-\delta(n-1)I_n} + \norm{C - \delta(\un_n\un_n^T - I_n)}.
\end{align*}
For $1 \le i \le n$, we have the following equality:
\begin{equation*}
    (\diag(C\un_n)-\delta(n-1))_{ii} = \sum_{j \neq i} (C_{ij} - \delta).
\end{equation*}
Let $h(u)=(1+u)\log(1+u)-u=\frac{u^2}{2}(1+o_{u\to0}(1))$. Using Bennett's inequality, we get for $t \ge 0$
\begin{align*}
  \mathbb{P}\left(\left|\sum_{j \neq i} (C_{ij} - \delta)\right| > t\right)
  & \le 2\exp{\left(-\frac{(n-1)(1-\delta)}{1+\delta}
    h\left(\frac{t}{(n-1)(1-\delta)}\right)\right)} \\
  & \le 2\exp{\left(-\left(\frac{n-1}{1+\delta}\right)
    h\left(\frac{t}{n-1}\right)\right)}.
\end{align*}
The second inequality is true because $h$ is convex and $h(0)=0$, so $ah(x/a)\geq h(x)$ for all $x\geq 0,a\in]0;1]$.

We set $t=\sqrt{(2+\varepsilon')(1+\delta)n\log n}$ for some $\varepsilon' <\frac{\varepsilon}{2}$. Observe that $\frac{t}{n-1} \to 0$ when $n\to +\infty$, so $h\left(\frac{t}{n-1}\right)=\frac{t^2}{2n^2}(1+o(1))$ and
\begin{align*}
  \mathbb{P}\left(\left|\sum_{j \neq i} (C_{ij} - \delta)\right|>\sqrt{(2+\epsilon')(1+\delta)n\log n}\right)
  & \le 2 n^{-(1+\varepsilon'/2)(1+o_{n\to\infty}(1))}.
\end{align*}
Therefore, using a union bound, we get
\begin{equation*}
    \mathbb{P}\left(\norm{\diag(C\un_n)-\delta(n-1)I_n} \leq \sqrt{(2+\varepsilon')(1+\delta)n\log n} \right) \ge 1 - 2 n^{-(\varepsilon'/2+o(1))}.
\end{equation*}
Moreover, from \citep*[Lemma 2]{mcrae2024nonconvex}, with probability at least $1-n^{-3}$,
\begin{equation*}
  \norm{C - \delta(\un_n\un_n^T - I_n)}  \lesssim \sqrt{n},
\end{equation*}
This bound is negligible in front of $\sqrt{n\log n}$, for $n$ large enough, so that
\begin{equation*}
    \norm{\bm{L}-\mathbb{E}(\bm{L})} < \sqrt{(2+2\varepsilon')(1+\delta)n\log n},
\end{equation*}
with probability at least $1-n^{-3}-2 n^{-\frac{\varepsilon'}{3}}$. Therefore, we get the desired result if
\begin{align*}
  \frac{\delta n + \sqrt{(2+2\varepsilon')(1+\delta)n\log n}}{\delta n - \sqrt{(2+2\varepsilon')(1+\delta)n\log n}} < p
  & \iff \delta > \frac{1+\sqrt{1+ 4\left(\frac{p-1}{p+1}\right)^2
    \frac{n}{(2+2\varepsilon')\log n}}}{2\left(\frac{p-1}{p+1}\right)^2 \frac{n}{(2+2\varepsilon')\log n}}.
\end{align*}
In the regime when $n$ is large, recalling that $\varepsilon>2\varepsilon'$, this is implied by
\begin{align*}
  \delta > \frac{p+1}{p-1} \sqrt{\frac{(2+\varepsilon) \log n}{n}}.
\end{align*}
\end{proof}
\begin{proof}[Proof of corollary~\ref{Kuracor}]
    We will prove the following more general result of which Corollary~\ref{Kuracor} is a consequence.
    \begin{cor*}
    We consider the Kuramoto model on the sphere $\s^{p-1}$ with coupling matrix as defined in \eqref{Ckur}, with $\alpha \in[0,1/2)$. For any $\varepsilon > 0$ and $n$ large enough, if
    \begin{equation*}
      \alpha < \frac{1}{2} - \frac{p+1}{2(p-1)} \sqrt{\frac{(2+\varepsilon) \log n}{n}},
    \end{equation*}
    then the Kuramoto system \eqref{GDKura} synchronizes with probability at least $1-n^{-3}-n^{-\frac{\varepsilon}{3}}$.
    \end{cor*}
    Note that the coupling matrix of the Kuramoto model \eqref{Ckur} is just the cost matrix of the $\mathbb{Z}_2$-synchronization model with Bernoulli noise \eqref{Cber} with the choice $\bm{x}=\un_n$ and $1-\alpha=\frac{1+\delta}{2}$, the rest follows.
\end{proof}

\subsection{Proofs of section~\ref{section4}}\label{proof_sec_4}
\begin{proof}[Proof of lemma~\ref{lemmeP1primal}]

Let $(L,\mu)$ be a solution of~\eqref{primal}. Since $P_{\perp}LP_{\perp} \succeq P_{\perp}$ it holds that $\lambda_2(P_{\perp}LP_{\perp}) \ge 1$, therefore $\frac{\lambda_n(P_{\perp}LP_{\perp})}{\lambda_2(P_{\perp}LP_{\perp})} \le \lambda_n(P_{\perp}LP_{\perp})$. Since $(P_{\perp}LP_{\perp},\mu)$ is feasible for~\eqref{P1}, the optimal value of~\eqref{P1} is less than that of~\eqref{primal}.

Now, let $(L,\mu)$ be feasible for~\eqref{P1} and define
\begin{equation*}
(L',\mu')=\left( \frac{L}{\lambda_2(L)}, \frac{\mu}{\lambda_2(L)} \right),
\end{equation*}
which is feasible for~\eqref{primal}. The last two constraints of~\eqref{primal} are easily verified. For the first constraint, note that $\operatorname{Ker}(L)=\un_n$; therefore, for all $x\in\Rn$, $P_{\perp}x$ is the projection of $x$ onto the orthogonal of $\operatorname{Ker}(L)$, and $\norm{Lx}^2 \geq \lambda_2(L) \norm{P_{\perp}x}^2$.
This implies that $P_{\perp}L'P_{\perp}\succeq P_{\perp}$. Thus we have

\begin{equation*}
    \mbox{Opt}\mbox{~\eqref{primal}} \le \lambda_n(L') = \frac{\lambda_n(L)}{\lambda_2(L)}.
\end{equation*}
By minimizing both sides of the inequality for all $L$ feasible for~\eqref{P1}, we get
\begin{equation*}
    \mbox{Opt}~\eqref{primal} \le \mbox{Opt}~\eqref{P1}.
\end{equation*}
\end{proof}
\begin{proof}[Proof of lemma~\ref{lemdual}]
    First, we first incorporate the constraints into the cost function and problem~\eqref{primal} becomes
\begin{align*}
  \underset{(L,\mu) \in \s^{n \times n} \times \Rn}{\mbox{inf}} \lambda_n(P_{\perp}LP_{\perp})
  & + \underset{Z \succeq 0}{\mbox{sup}} - \scal{P_{\perp}LP_{\perp}-P_{\perp}}{Z} \\
  & + \underset{W \in \R^{n \times p}}{\mbox{sup}} \scal{(P_{\perp}LP_{\perp}-\diag(\mu))V}{W} \\
  & + \underset{H \in K}{\mbox{sup}} \scal{\diag(\mu)-P_{\perp}LP_{\perp}}{H}.
\end{align*}
To lighten notations, define the constraint set $\mathcal{C}$ as :
\begin{equation*}
    \mathcal{C}=\{(W,Z,H)\in\Rnp \times \s^{n \times n} \times \s^{n \times n}:Z\succeq 0 \mbox{ and } H \in K\}.
\end{equation*}
We symmetrize and simplify the previous expression of problem~\eqref{primal}. We get that it is equal to
\begin{align*}
  \underset{(L,\mu) \in \s^{n \times n} \times \Rn}{\mbox{inf}}
  &  \lambda_n(P_{\perp}LP_{\perp}) \\
  & + \underset{(W,Z,H)\in\mathcal{C}}{\sup} - \scal{P_{\perp}\left(Z+H-\frac{WV^T+VW^T}{2}\right)P_{\perp}}{L} \\
  & \hskip 3cm + \scal{P_{\perp}}{Z} + \scal{\diag(\mu)}{H-WV^T} .
\end{align*}
By inverting the $\inf$ and the $\sup$ we get
\begin{equation}\label{maxmin}
    \begin{aligned}
      \mbox{Opt}& ~\eqref{primal} \\
                &\ge  \underset{(W,Z,H)\in\mathcal{C}}{\sup} \scal{P_{\perp}}{Z}  \\ & + \underset{L \in \s^{n \times n}}{\mbox{inf}} \lambda_n(P_{\perp}LP_{\perp}) - \scal{P_{\perp}\left(Z+H-\frac{WV^T+VW^T}{2}\right)P_{\perp}}{L} \\
    & + \underset{\mu \in \Rn}{\mbox{inf}} \scal{\diag(\mu)}{H-WV^T}.
    \end{aligned}
\end{equation}
The next step is to rewrite the last two terms of the right hand side of inequality~\eqref{maxmin} as characteristic functions of convex sets. Note that
\begin{equation*}
    \underset{\mu \in \Rn}{\mbox{inf}} \scal{\diag(\mu)}{H-WV^T} = \left\{
    \begin{array}{ll}
        0& \mbox{if } \diag(H)=\diag(WV^T) \\
        -\infty & \mbox{otherwise.}
    \end{array}
\right.
\end{equation*}
Moreover, by letting $M=P_{\perp}\left(Z+H-\frac{WV^T+VW^T}{2}\right)P_{\perp}$, we have
\begin{align*}
    \underset{L \in \s^{n \times n}}{\mbox{inf}} & \lambda_n(P_{\perp}LP_{\perp}) - \scal{M}{L} = \left\{
    \begin{array}{ll}
        0 & \mbox{if } M \succeq 0 \mbox{ and } \tr(M) \le 1, \\
        -\infty & \mbox{otherwise.}
    \end{array}
\right.
\end{align*}
To see this, assume first that $M$ is not positive semidefinite. Therefore, we can write the eigendecomposition of $M$ as $M=\sum_{i=1}^n\rho_iu_iu_i^T$ with $\rho_1 = 0$ and $u_1=\frac{\un_n}{\sqrt{n}}$ (since $\un_n$ belongs to the kernel of $M$) and $\rho_2 < 0$.
Take $L_x = xu_2u_2^T$, with $x < 0$. By noting that $P_\perp L_x P_\perp = L_x$, we have
\begin{equation*}
    \lambda_n(P_{\perp}L_xP_{\perp}) - \scal{M}{L_x} = -x\rho_2 \norm{u_2}^2 \xrightarrow[x \to - \infty ]{}-\infty.
\end{equation*}
Now, assume that $\tr(M) > 1$ and take $L_y=yI_n$ with $y>0$. We have
\begin{equation*}
    \lambda_n(P_{\perp}L_yP_{\perp}) - \scal{M}{L_y} = y(1-\tr(M)) \xrightarrow[y \to \infty ]{}-\infty.
\end{equation*}
Finally, assume that $M \succeq 0$ and $\tr(M) \le 1$. It is always true that for any symmetric matrix $L$, $P_\perp L P_\perp \preceq \lambda_n(P_\perp L P_\perp) I_n$. Therefore, since $M \succeq 0$, we have
\begin{align*}
    \scal{P_\perp L P_\perp}{M} & \le \scal{\lambda_n(P_\perp L P_\perp) I_n}{M} \\
    & = \lambda_n(P_\perp L P_\perp)\tr(M).
\end{align*}
Using the fact that $\tr(M)\leq 1$ and $\scal{P_\perp L P_\perp}{M} = \scal{L}{M}$, we get that $\lambda_n(P_\perp L P_\perp) - \scal{M}{L} \ge 0$ and the bound is reached for $L=0$. To conclude, the right hand side of inequality~\eqref{maxmin} becomes
\begin{equation*}
\begin{aligned}
\underset{(W,Z,H)\in \mathcal{C}}{\mbox{sup}} \quad & \left\langle Z, P_{\perp} \right\rangle\\
\textrm{s.t.} \quad & \diag(WV^T) = \diag(H) \\
  & M = P_{\perp}\left(Z + H -\frac{1}{2}\left(WV^T + VW^T\right)\right)P_{\perp} \\
  & M \succeq 0 \\
  & \tr(M) \le 1. \\
\end{aligned}
\end{equation*}
\end{proof}
\begin{proof}[Proof of lemma~\ref{lemb}]
    We assume $\beta\geq 0$ and Equation~\eqref{H2} holds. By developping and multiplying by $p-1$, this equation can be rewritten as
  \begin{equation}
    \beta\left(\frac{p}{\scal{P_{\perp}}{VV^T}} - (p-1)\beta\right) - \frac{\delta^2}{4}
    \geq 0. \label{H2_equiv}
  \end{equation}
  We have  
  \begin{align*}
    M_* & = P_{\perp}\left(Z_* + H_* -\frac{1}{2}\left(W_*V^T + VW_*^T\right)\right)P_{\perp} \\
        & = P_{\perp} \Bigg( \left( \frac{p}{\scal{P_\perp}{VV^T}}
          - (p-1)\beta\right)VV^T + \beta\underbrace{(VV^T)\odot (VV^T)}_{\substack{\succeq 
          (VV^T)\odot (v_1v_1^T) \\
    = \diag(v_1)VV^T\diag(v_1)^T}} \\
        & \qquad + \frac{\delta}{2}(\diag(v_1)VV^T + VV^T\diag(v_1) )  \Bigg)P_{\perp}  \\
        & \succeq P_{\perp} \Bigg( \left( \frac{p}{\scal{P_\perp}{VV^T}}
          - (p-1)\beta\right)VV^T + \beta \diag(v_1)VV^T\diag(v_1)^T \\
        & \qquad + \frac{\delta}{2}(\diag(v_1)VV^T + VV^T\diag(v_1) )  \Bigg)P_{\perp}  \\
        & = \begin{pmatrix}V&\diag(v_1)V\end{pmatrix}
          \left(S \otimes I_p\right)
          \begin{pmatrix}V&\diag(v_1)V\end{pmatrix}^T,
  \end{align*}
  where the notation ``$\otimes$'' stands for the Kronecker product and we have defined
  \begin{equation*}
    S = \begin{pmatrix}
      \frac{p}{\scal{P_\perp}{VV^T}}-(p-1)\beta & \frac{\delta}{2} \\
      \frac{\delta}{2} & \beta \\
    \end{pmatrix}.
  \end{equation*}
  \\
  All principal minors of $S$ are nonnegative. Indeed, $\beta\geq 0$ by assumption,
  \begin{align*}
    \det(S)
    & = \beta\left(\frac{p}{\scal{P_{\perp}}{VV^T}} - (p-1)\beta\right) - \frac{\delta^2}{4} \\
    & \geq 0\mbox{ from Equation~\eqref{H2_equiv},}
  \end{align*}
  and $\frac{p}{\scal{P_{\perp}}{VV^T}} - (p-1)\beta \geq 0$ (otherwise, we must have $\beta > \frac{p}{(p-1)\scal{P_{\perp}}{VV^T}}>0$, so $\det(S)< - \frac{\delta^2}{4}$, which contradicts the nonnegativity of the determinant). Therefore, $S\succeq 0$, so $S\otimes I_p\succeq 0$ and $M_*\succeq 0$.
\end{proof}

\bibliographystyle{plainnat}
\bibliography{references}

\end{document}